\documentclass{article}
\usepackage{amsthm}
\usepackage{kpfonts}
\usepackage{theoremref}

\newcommand{\cadlag}{\mbox{c\`adl\`ag}} %cadlag processes
\newcommand{\cad}{\mbox{c\`ad}}
\newcommand{\lag}{\mbox{l\`ag}}
\newcommand{\caglad}{\mbox{c\`agl\`ad}} %caglad processes
\newcommand{\cag}{\mbox{c\`ag}}
\newcommand{\lad}{\mbox{l\`ad}}
\def\reals{\mathbb{R}}
\def\uball{\mathbb{B}}
\def\ereals{\overline{\mathbb{R}}}
\def\interior{\mathop{\rm int}\nolimits}
\def\argmin{\mathop{\rm argmin}\limits}
\def\dom{\mathop{\rm dom}\nolimits}
\def\naturals{\mathbb{N}}
\def\cl{\mathop{\rm cl}\nolimits}
\def\co{\mathop{\rm co}}
\def\tos{\rightrightarrows}

\newcommand{\pathy}{\mathtt{y}} %a cadlag deterministic path y
\newcommand{\pathspace}{\mathsf{D}} % path space 
\def\epi{\mathop{\rm epi}}

\newtheorem{theorem}{Theorem}
\newtheorem{lemma}[theorem]{Lemma}
\newtheorem{corollary}[theorem]{Corollary}
\newtheorem{proposition}[theorem]{Proposition}
\newtheorem{example}{Example}
\newtheorem{remark}{Remark}
\theoremstyle{definition}
\newtheorem{definition}{Definition}
\newtheorem{assumption}{Assumption}

\begin{document}
\title{Michael selections and Castaing representations with  $\cadlag$ functions}
\author{Ari-Pekka Perkki\"o\thanks{Department of Mathematics, Ludwig Maximilians Universit\"at M\"unchen, Theresienstr. 39, 80333 M\"unchen, Germany.}
\and
Erick Trevi\~no-Aguilar\thanks{Instituto de Matem\'aticas, Unidad Cuernavaca, Universidad Nacional Aut\'onoma de M\'exico.}}

\maketitle

\begin{abstract}
Michael's selection theorem implies that a closed convex nonempty-valued mapping from the Sorgenfrey line to a euclidean space is inner semicontinuous if and only if the mapping can be represented as the image closure of right-continuous selections of the mapping. This article gives necessary and sufficient conditions for the representation to hold for cadlag selections, i.e., for selections that are right-continuous and have left limits. The characterization is motivated by continuous time stochastic optimization problems over cadlag processes. Here, an application to integral functionals of cadlag functions is given.
% The topology $\tau_r$ of the Sorgenfrey arrow is the paracompact perfectly normal topology under which continuous functions are the right continuous functions. Michael's selection theorems states that a closed convex nonempty-valued $\Gamma: [0,T] \tos \reals^d$ is $\tau_r$-inner semicontinuous if and only if it has the representation
% \[
% \Gamma_t = \cl \{\pathy(t) \mid \pathy  \in C(\Gamma)\},
% \]
% where $C(\Gamma)$ denotes the $\tau_r$-continuous selections of $\Gamma$.  This article gives necessary and sufficient conditions for  the representation to hold for cadlag selections, i.e., for selections that are right continuous and have left limits. An application is given to integral functionals of cadlag functions.
\end{abstract}


%---------------------------------------------------------
\section{Introduction} \label{sec:Intro}
%---------------------------------------------------------
The celebrated Michaels' selection theorem \cite[Theorem 3.2'']{mic56} characterizes in a $T_1$ topological space $(X,\tau)$ the property of paracompactness through the existence of a continuous selection for each mapping $\Gamma : X \tos Y$ which is convex closed valued in a fixed Banach space $Y$ and satisfies the minimal continuity condition of inner-semicontinuity (this was called lower semicontinuity in \cite{mic56}). Once the existence of a continuous selection is established it is natural to ask if 
\[
\Gamma_t = \cl \{\pathy(t) \mid \pathy  \in C(\Gamma,\tau)\},
\]
where $C(\Gamma,\tau)$ is the set of selections of $\Gamma$ which are continuous with respect to $\tau$. For this representation problem, \cite[Lemma 5.2]{mic56} provides a positive answer for perfectly normal topological spaces and separable Banach spaces.  

For the applications of stochastic processes and the construction of paths with specific properties, the case $X=[0,T]$ for $T>0$ and $Y=\reals^d$ is relevant.  Take a a convex closed-valued  mapping $\Gamma:[0,T] \tos \reals^d$. Consider the euclidean topology $\tau_e$ in  $\reals$ relativized to $[0,T]$  and assume that $\Gamma$ is  $\tau_e$ inner-semicontinuous. Then  \cite[Theorem 3.2'']{mic56} guarantees the existence of a $\tau_e$ continuous selection of $\Gamma$ and \cite[Lemma 5.2]{mic56} a representation with $\tau_e$ continuous selections.  Applied to the  Sorgenfrey arrow topology  $\tau_r$ relativized to $[0,T]$ and assuming that  $\Gamma$ is $\tau_r$ inner-semicontinuous, gives the representation
\[
\Gamma_t = \cl \{\pathy(t) \mid \pathy  \in C(\Gamma,\tau_r)\},
\]
where $C(\Gamma,\tau_r)$ denotes the family of selections of $\Gamma$ which are continuous with respect to $\tau_r$. Continuous functions in $\tau_r$ are right-continuous in the usual sense.

In many applications of stochastic processes it is usual to work with paths lying between continuous and right-continuous paths, namely, with right continuous paths having left limits  (abrv. $\cadlag$). In this case, the mapping $\Gamma$ having a representation with $\cadlag$ selections will satisfy other properties beyond inner-semicontinuity with respect to a given topology. The aim of this article is to formulate and characterize an equivalent property for a mapping having a representation with $\cadlag$ selections.

Our motivation comes from  continuous time stochastic optimization, especially stochastic singular control. In a series of articles, Rockafellar studied continuous selections and integral functionals of continuous functions and gave applications to convex duality in optimal control and in problems of Bolza; see the review article \cite{roc78}. The article \cite{pp18d} builds on Rockafellars results  and studies the stochastic setting of "regular" stochastic processes. Here, we extend the theoretical background to a deterministic setting with \cadlag\ functions. This forms a starting point of the companion paper \cite{pt21a} that deals with integral functionals of general \cadlag\ stochastic processes and allows us to go beyond the scope of \cite{pp18d}.  The follow-up papers \cite{pt22a,pt22b} give applications to finance and stochastic singular control. 
 
The rest of the article is organized as follows. In Section \ref{sec:notations}, we introduce notations, assumptions and basic concepts. Specially, here we formulate the main assumption that allows us to obtain existence of  $\cadlag$ selections and a representation of $\Gamma$ through them; see \thref{ass}.  In this same section we formulate our main result, the Theorem \ref{lab:maintheorem1}.  In Section \ref{sec:necessaryconditions}, we prove that the conditions in Theorem \ref{lab:maintheorem1} are necessary for a representation of a mapping $\Gamma$ in terms of its $\cadlag$ selections. The most relevant part being that a representation with $\cadlag$ selections implies \thref{ass}. In Section \ref{sec:pointscontinuity}, we show that $\Gamma$ coincides outside a countable set with its mapping of left limits. This result will allow us to obtain a $\cadlag$ selection which is continuous outside a  countable set depending only on the mapping. In Section  \ref{sec:cadlagselectionrepresentation}, we prove two results that yield the sufficiency in Theorem  \ref{lab:maintheorem1}, these results are Theorem \ref{lab:maintheorem2} and Proposition \ref{labprop:rep}. In Section \ref{sec:examples}, we illustrate with examples our main result, Theorem \ref{lab:maintheorem1}.  In Section \ref{lab:sectionApp}, we give an application to integral functionals of cadlag functions. 
%---------------------------------------------------------
\section{Notations and main theorem} \label{sec:notations}
%---------------------------------------------------------
Let $\uball$ be the open unit ball of $\reals^d$, $T>0$, and $\pathspace$ the class of $\cadlag$ functions $\pathy:[0,T] \to \reals^d$.  We denote by $\tau_r$  the topology on $\reals$ generated by the intervals of the form $[a,b)$ with $a<b$, and by $\tau_l$ the topology on $\reals$ generated by the intervals $(a,b]$. Throughout the article, we fix a convex-valued mapping $\Gamma : [0,T] \tos \reals^d$. The set
\[
\dom \Gamma:=\{t\in [0,T] \mid \Gamma_t \neq \emptyset\}
\]
is the {\em domain} of $\Gamma$. A set-valued mapping is said to have {\em full domain} if its domain is the whole $[0,T]$.
 
The set $\pathspace(\Gamma)$ is the class of functions  $\pathy \in \pathspace$ such that $\pathy(t) \in \Gamma_t$ for all  $t \in [0,T]$. In other words, $\pathspace(\Gamma)$ is the class of $\cadlag$ selections of $\Gamma$. In this article, our main result establishes an equivalent condition for the validity of the representation
\begin{equation} \label{eq:MichaelRepresentation}
\Gamma_t = \cl \{\pathy(t) \mid \pathy  \in \pathspace(\Gamma)\}.
\end{equation}
We call \eqref{eq:MichaelRepresentation} the $\cadlag$ representation of $\Gamma$. If the representation holds, Proposition~\ref{labprop:rep} below shows that there exists a countable family $\{ \pathy_{\nu}(t)\}_{\nu \in \naturals}$ of $\cadlag$ selections with
\[
\Gamma_t = \cl \{\pathy_{\nu}(t), \nu \in \naturals \}.
\] 
 If this representation holds for selections that are merely measurable, we arrive at a {\em Castaing representation} of $\Gamma$; see, e.g., \cite{rw98}. Hence, we obtain a Castaing representation with $\cadlag$ functions.

In the next definition we recall inner-semicontinuity  that is necessary for the representation \eqref{eq:MichaelRepresentation}. Example \ref{labEx1} below shows it is not sufficient.
\begin{definition}\label{labDefisc}
	A mapping $\phi$ is inner-semicontinuous with respect to the relative $\tau_r$ topology in $[0,T]$,  if for each open $O \subset \reals^d$, the set
	\[
	\phi^{-1}(O):=\{t \in [0,T] \mid O \cap \phi_t \neq \emptyset\}
	\]
	is the intersection of $[0,T]$ with a $\tau_r$-open set, or briefly, it is relatively $\tau_r$-open.  This property will be denoted by $\tau_r$-isc.
\end{definition}

\begin{example}\label{labEx1}
The mapping $\phi: [0,\pi] \tos \reals$ defined for $t \in [0,\pi)$ by  $\phi_t:=\{\sin(1/(\pi-t))\}$  and $\phi_{\pi}=\{2\}$ shows that the property in Definition \ref{labDefisc} is not sufficient for  the representation \eqref{eq:MichaelRepresentation}.  Indeed, $\phi$ is  $\tau_r$-isc but  \eqref{eq:MichaelRepresentation} fails.
\end{example}

As suggested by the previous example, left sided limits of a mapping are an essential element to the characterization of the representation \eqref{eq:MichaelRepresentation}. They are defined as follows. For a mapping $\phi:[0,T] \tos \reals^d$ let $\vec\phi_0:=\{0\}$ and
\begin{align*}
\vec\phi_t
&:= \liminf_{s \upuparrows t} \phi_s=\bigcap_{\{t_n\}_{n \in \naturals} } \liminf_{n \to \infty} \phi_{t_n},
\end{align*}
where the limits are in the sense of \cite[Section 5.B]{rw98} and the intersection is over all strictly increasing sequences $\{t_n\}_{n \in \naturals}$ converging to $t$.  Consistent with $\vec\phi_0:=\{0\}$ we define for a $\cadlag$ function $\pathy$ the left limit $\pathy(0-):=0$.

We state a basic property of $\vec \Gamma$ which follows directly from its definition. It is however quite useful and we formulate it as a lemma for reference. Note the roles of $\vec\Gamma^{-1}$ and   $\Gamma^{-1}$.
\begin{lemma}\label{lem:Elementary1}
Assume the mapping $\vec \Gamma$ has a full domain.  Let $O  \subseteq \reals^d$ be open with $\vec\Gamma^{-1}(O)\neq \emptyset$.  For $t \in \vec\Gamma^{-1}(O)\cap (0,T]$ there exists $\delta>0$ such that $(t-\delta,t) \subset \Gamma^{-1}(O)$.
\end{lemma}
\begin{proof}
Assume for a contradiction that an strictly increasing sequence $\{t_n\}_{n\in \naturals} \subset [0,T] \setminus \Gamma^{-1}(O)$ exists  converging to $t$.  Take $y \in O\cap \vec \Gamma_t$ and $\eta>0$ such that $y + 2\eta \uball \subset O$.  Then $d(y,\Gamma_{t_n}) \geq \eta$ and there is no sequence $\{y_n\}_{n \in \naturals}$ with $y_n \in \Gamma_{t_n}$ converging to $y$, a contradiction with the definition of $\vec \Gamma$. 
\end{proof}

We show in Lemma \ref{labLemma1} below that $\dom \vec \Gamma=[0,T]$, where $\dom \vec \Gamma$ denotes the  domain of $\vec \Gamma$,  is sufficient for the existence of selections of the ``$\epsilon$-fattening'' $\Gamma + \epsilon \uball$. We call selections of the $\epsilon$-fattening as  $\epsilon$-selections. Note that the mapping $\phi$ in Example \ref{labEx1} fails this property since $\vec \phi$ is empty at $T$.  For actual selections we verify in Proposition \ref{labprop:neccesaryA} below that the stronger property of the next assumption is necessary. In Theorem \ref{lab:maintheorem2}  and Proposition \ref{labprop:rep} we will prove it is also sufficient.
\begin{assumption}\thlabel{ass}
\label{labAssA}
For every $t \in (0,T]$ and bounded open set  $O\subset \reals^d$,
\[
(t-\delta,t) \subset \Gamma^{-1}(O) \mbox{ for some } \delta>0 \implies   \overrightarrow{\Gamma\cap O} \neq \emptyset \mbox{ at } t.
\]
\end{assumption}
The assumption rules out ``oscillations from the left''. For instance, the wildly oscillating mapping $\phi$ in Example~\ref{labEx1} does not satisfy the above assumption.  

\begin{remark}\label{labRemwAssA}
Note that \thref{ass} implies $t\in \vec\Gamma^{-1}(\cl O)$, but  $t\in \vec\Gamma^{-1}(O)$ does not need not hold. Indeed, defining
\[
\phi:= \begin{cases}
[0,1-t] & \mbox{ for } t\in [0,1)\\
\{2\} & \mbox{ for } t=1,
\end{cases}
\] 
and choosing $O=(0,1)$, we have $\phi^{-1}(O)=[0,1)$ while  $1 \notin  \vec \phi^{-1}(O)$ and $1 \in  \vec \phi^{-1}(\cl O)$.  Note that $\phi$ has the representation \eqref{eq:MichaelRepresentation}. 
\end{remark}

The next theorem is the main result of the article. The necessity is established in Proposition \ref{labprop:neccesaryA} while the sufficiency is obtained from Theorem \ref{lab:maintheorem2} and  Proposition~\ref{labprop:rep}.

\begin{theorem}\label{lab:maintheorem1}
Assume that $\Gamma$ is closed convex-valued with full domain. Then 
\begin{equation*}
\Gamma_t = \cl \{\pathy(t) \mid \pathy \in \pathspace(\Gamma)\}
\end{equation*}
if and only if  $\Gamma$ is $\tau_r$-isc,  $\vec \Gamma$ has full domain and \thref{ass} holds. In this case,
\begin{equation*}
\Gamma_t = \cl \{\pathy(t) \mid \pathy \in \pathspace(\Gamma)  \mbox{  continuous on } [0,T]\setminus D_1\},
\end{equation*}
where $D_1$ is the countable set defined by
\[
D_1:=\{t \in (0,T] \mid \Gamma_t  \not \subset \vec\Gamma_t\}.
\]
\end{theorem}

%---------------------------------------------------------
\section{Necessity of the $\cadlag$ representation} \label{sec:necessaryconditions}
%---------------------------------------------------------
It is clear that the cadlag representation \eqref{eq:MichaelRepresentation} implies that $\Gamma$ is a closed-valued mapping with full domain and that $\Gamma$ is $\tau_r$-isc. Furthermore, $\vec \Gamma$ must have full domain. It is less obvious that the representation \eqref{eq:MichaelRepresentation} yields \thref{ass}.
\begin{proposition}\label{labprop:neccesaryA}
If the representation \eqref{eq:MichaelRepresentation} holds, then $\Gamma$ satisfies \thref{ass}.
\end{proposition}
\begin{proof}
Take $t \in (0,T]$ and an open set  $O\subset \reals^d$. Assume that   
$[t-\delta,t) \subset \Gamma^{-1}(O)$ for some  $\delta>0$. For $z \in (t-\delta,t) $ denote by $\pathspace(\Gamma\cap O,z)$ the $\cadlag$ selections $\pathy \in \pathspace(\Gamma)$ with the further property that $\pathy(s) \in O$  for $s \in [t-\delta,z)$. Define the set
\[
A:=\{z \in  (t-\delta,t) \mid  \pathspace(\Gamma \cap O,z) \neq \emptyset \}.
\]
Claim: $A$ is non-empty and for $z^*:=\sup A$ we have $z^*=t$. After the claim there exists a $\cadlag$ function $\pathy \in \pathspace(\Gamma)$ with $\pathy(z) \in \Gamma_z \cap O$ for $z \in [t-\delta,t)$. Then, $\pathy(t-)\in \vec \Theta_t$ where $\Theta=  \Gamma \cap O$ proving the proposition.
	
Now we verify the claim. Take $y \in \Gamma_{t-\delta}\cap O$ and $\epsilon>0$ such that $y + \epsilon \uball \subset O$. Let $\pathy \in D(\Gamma)$ be such that $\pathy(t-\delta) \in \Gamma_{t-\delta}\cap (y + \epsilon \uball)$. There exists $\eta\in (0,\delta)$ such that $\pathy(z) \in y + \epsilon \uball$ for $z \in [t-\delta,t-\delta+\eta)$ since $\pathy$ is right continuous. Then, $\pathy \in D(\Gamma \cap O, t-\delta+\eta)$ showing that $(t-\delta+\eta) \in A$. Thus, $A$ is nonempty. Now assume that $z^*<t$, let $\pathy \in \pathspace(\Gamma\cap O,z^*)$. By the same argument as before we can find a function $\mathsf{g} \in \pathspace(\Gamma)$ with $\mathsf{g}(z) \in \Gamma_z\cap O$ for $z \in [z^*,z^*+\delta')$ and some $\delta'>0$. Thus, $(z^*+\delta')\wedge t \in A$ with function $\pathy 1_{[0,z^*)} + \mathsf{g}1_{[z^*,T]}$, contradicting the definition of $z^*$. Hence $z^*=t$ and the proof is complete.
\end{proof}

%---------------------------------------------------------
\section{Inner-semicontinuity from the left and right} \label{sec:pointscontinuity}
%---------------------------------------------------------
In this section we assume that $\Gamma$ is $\tau_r$-isc and  closed-valued. We show that $\Gamma$ is equal to $\vec \Gamma$ except for a countable set. In particular $\Gamma$, being a $\tau_r$-isc mapping is also isc from both sides, or more precisely $\tau_e$-isc, except for a countable set. This property will allow us to obtain a $\cadlag$ selection which is continuous outside a countable set depending only on $\Gamma$. 

We start with the next lemma showing that  $\Gamma_t$ is a subset of $\vec \Gamma_t$ for $t$ outside a countable set depending only on $\Gamma$.
\begin{lemma}\label{lem:D1}
The set 
\[
D_1:=\{t \in (0,T] \mid \Gamma_t  \not \subset \vec\Gamma_t\}.
\]
is countable.  
\end{lemma}
\begin{proof}
Let $\mathcal{A}$ be a countable family of open sets generating the euclidean topology $\tau_e$ of $\reals^d$.  For all $t \in D_1$ there exists $A_t \in \mathcal{A}$ with $A_t \cap  \Gamma_t  \neq  \emptyset$ and a strictly increasing sequence $z_n \upuparrows t$  such that $\cl A_{t} \cap \Gamma_{z_n} = \emptyset$. We check this claim. By way of contradiction assume that for each $A \in \mathcal{A}$ with $A \cap \Gamma_t \neq \emptyset$  it happens that for each strictly increasing sequence $z_n \upuparrows t$ there exists $n_A \geq 0$ such that for $n \geq n_A$ we have $\cl A \cap \Gamma_{z_n} \neq \emptyset$. Let us check that in this case we have $\Gamma_t \subset \vec \Gamma_t$. To this end, take $y \in \Gamma_t$. For $\epsilon>0$  let $A_\epsilon \in \mathcal{A}$ be such that $y \in A_{\epsilon}$ and $\cl A_{\epsilon} \subset y + \epsilon \uball$. For a sequence $z_n \upuparrows t$ there exists $n_{\epsilon} \geq 0$ such that for $n \geq n_{\epsilon}$ we have $\cl A_{\epsilon} \cap \Gamma_{z_n} \neq \emptyset$. Thus, it is possible to construct a sequence $\{y_n\}_{n \in \naturals}$ converging to $y$ with $y_n \in \Gamma_{z_n}$. Hence $y \in \liminf_{n} \Gamma_{z_n}$. It follows that $y \in  \vec  \Gamma_{t}$ since the sequence was arbitrary.

Now assume for a contradiction that  $D_1$ is uncountable. Then, there exists an infinite subset  $D_0 \subset D_1$ and $\tilde A \in \mathcal{A}$ such that $A_t=\tilde A$  for all $t \in D_0$. Moreover, there exists a point $t_0 \in D_0$ and a strictly decreasing sequence $\{t_n\}_{n \in \naturals} \subset D_0$ converging from the right to $t_0$, since $D_1$ is uncountable. To verify the latter elementary fact, if there are no right limit points of $D_1$ then for each $t \in D_1\setminus \{T\}$ there exists $\delta_t>0$ such that $(t,t+\delta_t) \subset [0,T]$ and $(t,t+\delta_t) \cap D_1 =\emptyset$. This produces a summable uncountable series of strictly positive numbers which is impossible.

As a consequence, for each $n$ there exists a strictly increasing sequence $\{z_{n,k}\}_{k\in \naturals}$ converging  to $t_n$ with $\cl \tilde A \cap  \Gamma_{z_{n,k}} = \emptyset$, $\tilde A \cap \Gamma_{t_n} \neq \emptyset$ and $t_{n+1}<z_{n,k}<t_{n}$. Hence $z_{n,1} \downdownarrows t_0$ while $\cl \tilde A \cap \Gamma_{z_{n,1}} = \emptyset$ and $\tilde A \cap \Gamma_{t_0}\neq \emptyset$, in contradiction to $\Gamma$ being $\tau_r$-isc. 
\end{proof} 

The next lemma has similar proof to Lemma \ref{lem:D1} and can be skipped.
\begin{lemma}\label{lem:D2}
The set  $\{t \in (0,T] \mid \vec\Gamma_t \not \subset  \Gamma_t\}$ is countable.  
\end{lemma}
\begin{proof}
Let $\mathcal{A}$ be a countable family of open sets generating the euclidean topology of $\reals^d$.  For all $t$ in the set  $\{t \in (0,T] \mid \vec\Gamma_t \not \subset \Gamma_t\}$  there exists $A_t \in \mathcal{A}$ with $A_t \cap  \vec\Gamma_t  \neq  \emptyset$ while $\cl A_t \cap \Gamma_t=\emptyset$. Assume for a contradiction that the set $\{t \in (0,T] \mid \vec\Gamma_t \not \subset \Gamma_t\}$  is uncountable. Then there exists an infinite subset $D_0$ and $\tilde A \in \mathcal{A}$ such that for all $t \in D_0$ we have $A_t=\tilde A$ and moreover there exists $t_0 \in D_0$ that is approximated from the left by an increasing sequence $\{t_n\}_{n \in \naturals} \subset D_0$. Then, for $n$ large enough we have $\tilde A  \cap \Gamma_{t_n} \neq \emptyset$ since $\tilde A \cap \vec\Gamma_{t_0}  \neq \emptyset$. This is a  clear contradiction with the properties of $\tilde A$ which is equal to $A_t$ for $t \in D_0$.
\end{proof}

The following is an immediate consequence of the above two lemmas.
\begin{corollary}\label{lem:D3}
The set  $\{t \in (0,T] \mid \vec\Gamma(t) \neq \Gamma(t) \}$ 
is countable.  
\end{corollary}

%---------------------------------------------------------
\section{Sufficiency of the $\cadlag$ representation} \label{sec:cadlagselectionrepresentation}
%---------------------------------------------------------
The domain of $\vec \Gamma$ includes $(0,T]\setminus D_1$ which is the complement of a countable set by Lemma \ref{lem:D1}. In the next lemma we assume that $\vec \Gamma$ has full domain  and  show that it is already strong enough for  $\epsilon$-selections. Note that we do not require $\Gamma$ to be closed-valued.
\begin{lemma}\label{labLemma1}
Assume $\Gamma$ is convex valued and $\tau_r$-isc, and that $\vec \Gamma$ has full domain. Then, for each $\epsilon>0$ there exists a $\cadlag$ selection of $\Gamma + \epsilon \uball$ that is  continuous on $[0,T]\setminus D_1$.
\end{lemma}
\begin{proof}
For $t \in (0,T)\setminus D_1$, take $y_t \in \Gamma_t \cap \vec \Gamma_t$. There exists $\delta>0$ such that $(t-\delta,t+\delta) \subset (0,T)$ and $(t-\delta,t+\delta) \subset \Gamma^{-1}(y_t + \epsilon \uball)$ due to the fact that $\Gamma$ is $\tau_r$-isc and also by Lemma \ref{lem:Elementary1}  since  $t\notin D_1$. Hence, the function $\pathy^t:=y_t 1_{(t-\delta,t+\delta)}$ is a local continuous selection of $\Gamma + \epsilon \uball$.

For $t  \in D_1\setminus \{0,T\}$  there exists  $y^r_t \in \Gamma_t$  since $\Gamma$ has full domain  and there exists $\delta_1>0$ such that  $[t,t+\delta_1) \subset (0,T)$
and $y^r_t \in \Gamma(z) + \epsilon \uball$ for $z \in [t,t+\delta_1)$ since $\Gamma$  is $\tau_r$-isc.  Considering that $\vec \Gamma$ has full domain take $y^l_t \in \vec \Gamma_t$. There exists $\delta_2>0$ such that $(t-\delta_2,t) \subset  \Gamma^{-1}(y^l_t + \epsilon \uball)\cap (0,T)$ due to Lemma \ref{lem:Elementary1}. The function $\pathy^t:=y^l_t 1_{(t-\delta_2,t)} + y^r_t 1_{[t,t+\delta_1)}$ is a local $\cadlag$ selection of $\Gamma + \epsilon \uball$.

For $t=0$ we can construct by similar arguments a local continuous selection. For $t=T$ a local selection exists that will be continuous or $\cadlag$ according to $T \notin D_1$ or $T \in D_1$.

The constructed intervals define an open covering (in the relative euclidean topology $\tau_e$) of the interval $[0,T]$. There exists a $\tau_e$-continuous partition of unity subordinated to a locally finite subcovering from which a global $\cadlag$ selection is generated with the required property of continuity outside $D_1$.
\end{proof}

Lemma \ref{labLemma1} provides $\epsilon$-selections. For selections we require the stronger condition, \thref{ass}. We need a preliminary result.
\begin{lemma}\label{lab:Lemreg1}
Assume $\Gamma$ is $\tau_r$-isc.  Let $\mathsf{g}$ be a $\cadlag$ selection of $\Gamma + \epsilon \uball$ which is continuous in $[0,T]\setminus D_1$. Let $\phi:=\Gamma \cap (\mathsf{g} + \epsilon \uball)$. Take $t \in (0,T]\setminus D_1$. Then for each $y \in \phi_t$ and $\eta>0$ there exists $\delta=\delta(y,\eta)>0$ such that $t-\delta>0$ and 
\[
y \in (\Gamma_z+\eta \uball) \cap (\mathsf{g}(z) + \epsilon \uball), \mbox{ for } z \in [t-\delta,t).
\]
\end{lemma}
\begin{proof}
Take $t \notin D_1$ with $t>0$, and $y \in \phi_t$. Let $\epsilon' \in (0,\epsilon)$ be such that  $\left| y-\mathsf{g}(t)\right| <\epsilon'$.
For $\rho:=(\epsilon-\epsilon')$ let $\delta_1>0$ be such that $t-\delta_1>0$ and for $z \in [t-\delta_1,t)$ we have $\left|\mathsf{g}(z)-\mathsf{g}(t)\right|<\rho$. There exists $\delta_2 \in (0,\delta_1)$  such that $[t-\delta_2,t) \subset \Gamma^{-1}(y+\eta \uball)$ due to Lemma \ref{lem:Elementary1}, since $y \in \vec \Gamma_t$.

For $z \in [t-\delta_2,t)$ let $r_1=y-\mathsf{g}(t) \in \epsilon' \uball$ and $r_2=\mathsf{g}(t)-\mathsf{g}(z)\in \rho \uball$. Then $y-r_2=\mathsf{g}(z)+r_1$ showing that
\[
y \in (\Gamma_z+\eta \uball) \cap (\mathsf{g}(z) + \epsilon \uball),
\]
which completes the proof.
\end{proof}

The next result is a Michael selection theorem for $\cadlag$ functions. Its proof proceeds by induction just like the proof  \cite[Theorem 3.2'']{mic56}, which is the original Michael selection theorem. 
\begin{theorem}\label{lab:maintheorem2}
Assume that $\Gamma$ is a $\tau_r$-isc convex-valued mapping with full domain, and that $\vec \Gamma$ has a full domain. Then, under \thref{ass}, there exists a $\cadlag$ selection of $\cl \Gamma$ which is continuous on $[0,T] \setminus D_1$.
\end{theorem}
\begin{proof}
Let $\epsilon_0=1$ and 	$\epsilon_i:=\frac{1}{2^i} \epsilon_{i-1}$. By Lemma~\ref{labLemma1}, there exists a $\cadlag$ selection $\pathy_1$ of $\Gamma + \epsilon_1 \uball$ continuous outside $D_1$. Assume $\pathy_1,\ldots,\pathy_k$ have been constructed with the following properties:
\begin{align*}
&(a)\ \pathy_i \mbox{ is a selection of } \pathy_{i-1} + 2 \epsilon_{i-1}\uball,\quad i=2,3,\ldots,k,\\
&(b)\ \pathy_i \mbox{ is a selection of } \Gamma + \epsilon_{i} \uball,\\
&(c)\ \pathy_i \mbox{ is continuous outside } D_1.
\end{align*}
Now we construct a function satisfying (a)-(c) for $i=k+1$. This will produce a sequence $\{\pathy_i\}_{i \in \naturals}$ converging under the uniform norm. Hence, it converges to a $\cadlag$ function $\pathy \in \pathspace$ which is continuous outside $D_1$. The function $\pathy$ is then a  selection of $\cl \Gamma$.

Let $\Gamma^{k+1}:=\Gamma \cap (\pathy_k + \epsilon_k \uball)$. The mapping $\Gamma^{k+1}$ is $\tau_r$-isc since it is the intersection of $\tau_r$-isc mappings; see Lemma \ref{lemma:iscStab}. It is clearly convex valued.  Moreover, $\dom \Gamma^{k+1}=[0,T]$ since $\pathy_{k}$ is a selection of $\Gamma + \epsilon_{k} \uball$.

Take $t \in (0,T) \setminus D_1$  and $y_t \in  \Gamma_{t}^{k+1}$. Let $\eta:=\frac{1}{2}\epsilon_{k+1}$. There exists $\delta_1>0$ such that $t-\delta_1>0$ and $y_t \in  (\Gamma_{z} + \eta \uball)\cap (\pathy_k(z) + \epsilon_k \uball) $ for $z \in (t-\delta_1,t)$, due to Lemma \ref{lab:Lemreg1}. There exists $\delta_2>0$ such that $t+\delta_2<T$ and $(y_t + \eta \uball)\cap \Gamma_z^{k+1}\neq \emptyset$ for $z \in [t,t+\delta_2)$, since $\Gamma^{k+1}$ is $\tau_r$-isc. Hence, for  $I_t:=(t-\delta_1,t+\delta_2)$ we have $I_t \subset (0,T)$ and  for $z \in I_t$ there exist $\gamma_z \in \Gamma_z$, $r_1, r_3 \in \eta \uball$, $r_2 \in \epsilon_k$  such that 
\[
\gamma_z + r_1=\pathy_k(z) + r_2=y_t + r_3.
\]
Thus, $y_t \in (\Gamma_z +  \epsilon_{k+1} \uball) \cap (\pathy_k(z)+2 \epsilon_{k} \uball)$. As a consequence, the function $\pathy^t_{k+1}:=y_t 1_{I_t}$ is a local continuous selection of $\Gamma + \epsilon_{k+1} \uball$ and $\pathy_k + 2 \epsilon_k \uball$.

Now take $t\in D_1 \cap (0,T)$. Let $\eta:=\frac{1}{3}\epsilon_{k+1}$ and $O_1:=\pathy_k(t-)+(\epsilon_{k} + \eta) \uball$. There exists $\delta_1>0$ such that $t-\delta_1>0$ and  $(t-\delta_1,t) \subset \Gamma^{-1}(O_1)$ since $\pathy_k$ is a selection of $\Gamma + \epsilon_k \uball$. Hence, there exists $y^l_t \in \overrightarrow{\Gamma_t \cap O_1}$ since $\Gamma$ satisfies \thref{ass}. 
There exists $\delta_2<\delta_1$ such that for $z \in (t-\delta_2,t)$ we have 
\[
\Gamma_z \cap O_1 \cap (y^l_t + \eta \uball) \neq \emptyset,
\]
due to Lemma \ref{lem:Elementary1} and we can select $\delta_2$ so that $|\pathy_k(t-)-\pathy_k(z)|<\eta$ for $z \in (t-\delta_2,t)$. Thus,  $y^l_t \in (\Gamma_z + \epsilon_{k+1} \uball) \cap (\pathy_k(z) + 2 \epsilon_{k} \uball)$ for $z \in (t-\delta_2,t)$. 
Take $y_t^r \in \Gamma^{k+1}_t$. There exists $\delta_3>0$ such that $t+\delta_3 \leq T$ and  $[t,t+\delta_3) \subset (\Gamma^{k+1})^{-1}(y_t^r + \eta  \uball)$ due to the fact that $\Gamma^{k+1}$ is $\tau_r$-isc. 
Let $I_t:=(t-\delta_2,t+\delta_3)$ and define a function $\pathy^t_{k+1}$ on $I_t$ by $\pathy^t_{k+1}:=y^l_t 1_{(t-\delta_2,t)} + y^r_t 1_{[t,t+\delta_3)}$. It is clear that $\pathy^t_{k+1}$ is a local $\cadlag$ selection of $\Gamma^{k+1} + \epsilon_{k+1} \uball$ and $\pathy_k + 2 \epsilon_k \uball$. For $t\in\{0,T\}$ we can construct local selections by similar arguments.

The constructed intervals $\{I_t\}_{t \in [0,T]}$ define an open covering (in the relative euclidean topology $\tau_e$) of the interval $[0,T]$. There exists a $\tau_e$-continuous partition of unity subordinated to a locally finite subcovering from which a global selection $\pathy_{k+1}$ can be produced by pasting together the local selections.
\end{proof}

Now we establish the sufficiency in Theorem~\ref{lab:maintheorem1}.
\begin{proposition}\label{labprop:rep}
Assume that $\Gamma$ is a $\tau_r$-isc closed convex-valued mapping  with full domain, and that  $\vec \Gamma$ has full domain. Under \thref{ass}, there exists a countable family $\{\pathy_{\nu}\}_{\nu \in \naturals} \subset \pathspace(\Gamma)$ of $\cadlag$ selections which are continuous in $[0,T]\setminus D_1$ and for $t \in [0,T]$
\[
\Gamma_t = \cl \{\pathy_{\nu}(t)\}_{\nu \in \naturals}.
\] 
\end{proposition}
\begin{proof}
If $T \in D_1$  we take a selection $\pathy \in \pathspace(\Gamma)$ continuous outside $D_1$ from which we can easily construct a sequence $\{\pathy_{j}\}_{j \in \naturals} \subset \pathspace(\Gamma)$ by modifying $\pathy$ at $T$ such that $\{\pathy_j(T)\}_{j \in  \naturals}$ is dense in $\Gamma_T$. This settles down the representation for $t=T$ in case $T \in D_1$.
	
Let $\mathcal{D}=\{y_m\}_{m\in \naturals}$ be a countable dense subset of $\reals^d$. Take $\epsilon=\frac{1}{2^k}$ for $k  \in \naturals$ and $y =y_m$ for $y_m\in \mathcal{D}$. Assume that $U:=\Gamma^{-1}(y + \epsilon \uball)$ is non empty.  The set $U$ is open in $\tau_r$ (relativized to $[0,T]$) and can be expressed as a countable union of intervals of the form $[a,b]$. Indeed, $\tau_r$ is hereditary Lindel\"of  and $U$ can be written as the countable union of intervals $[x,y)\cap [0,T]$ and each of these intervals can themselves be written as the countable union of intervals $[a,b]$. If $T \in U$ we distinguish two cases.  If $T \notin D_1$ then $T$  can be taken as an element of an interval $[a,T]\subset U$ due to Lemma \ref{lem:Elementary1}. In this case,  we will consider the representation (i) $U= \bigcup_{n \in \naturals} [a_n,b_n]$ with $a_n<b_n$, thus, no interval collapses to a point.  In the second case  $T \in U$ and $T \in D_1$ and we consider the representation (ii)  $U= \{T\} \cup  \bigcup_{n \in \naturals} [a_n,b_n]$ with $a_n<b_n<T$ and $b_n \notin D_1$. Hence, on both cases we do not consider trivial intervals and $b_n \notin D_1$.

Now take an interval $[a,b] =[a_n,b_n] \subset U$ with $a_n<b_n$ and   $b_n \notin D_1$. We fix the notation $[a,b]^c:=[0,T]\setminus [a,b]$. Let $\phi$ be the mapping defined by 
\[
\phi_z=
\begin{cases}
\Gamma_z & \mbox{ for } z \in [a,b]^c\\
\Gamma_z \cap (y + \epsilon \uball) & \mbox{ for } z \in [a,b].
\end{cases}
\]
It is simple to verify that the mapping $\phi$ is convex valued, has full domain,  and that it is $\tau_r$-isc.

We verify \thref{ass} and that  $\vec \phi$ has full domain.
Let $\delta>0$  and $t \in (0,T]$ be such that $(t-\delta,t)\subset \phi^{-1}(O)$ for $O$ an open subset of $\reals^d$.
For $t\notin (a,b]$ there exists $\delta'\in (0,\delta)$ such that $(t-\delta',t) \subset [a,b]^c$ so $\phi=\Gamma$ in this interval and we have $\phi \cap O=\Gamma \cap O$, hence  \thref{ass} is clearly satisfied in $(a,b]$.
Take now $t \in (a,b]$. There exists $\delta'' \in (0,\delta)$ such that $(t-\delta'',t)\subset (a,b]  \cap (t-\delta,t)$ and for  $z \in (t-\delta'',t)$
we have that $\phi_z \cap O =\Gamma_z \cap (y+\epsilon \uball) \cap O \neq \emptyset$ so by \thref{ass} we have $\overrightarrow{\Gamma \cap(y+\epsilon \uball) \cap O}=\overrightarrow{\phi \cap O}$ is non empty at $t$.  Thus, $\phi$ satisfies \thref{ass} and $\dom \vec \phi=[0,T]$.

The set $\{t \mid \phi_t \not \subset \vec \phi_t\}$ is included in $D_1$. Indeed, 
for $t\notin (a,b]$ there exists $\delta>0$ such that $(t-\delta,t) \subset [a,b]^c$ so $\vec \phi_t = \vec\Gamma_t$ and $\phi_t \subset \Gamma_t$. Hence, $\phi_t \subset \vec \phi_t$ whenever $t \notin D_1$. Take now $t \in (a,b] \setminus D_1$ and $w \in \phi_t$. Let $\eta>0$ be such that $w +\eta \uball \subset y + \epsilon \uball$. Note that $w \in \vec \Gamma_t$ since $t \notin D_1$ and then  $w \in \Gamma_t \subset \vec \Gamma_t$.
There exists   $\delta_{\eta} >0$ such that for $z \in (t- \delta_{\eta},t)$ we have  $\Gamma_z \cap (w +\eta \uball) \neq \emptyset$  due to Lemma \ref{lem:Elementary1},  since $w \in \vec \Gamma_t$. Hence, $\overrightarrow{\Gamma \cap(w+\eta \uball)} \neq \emptyset$ at $t$ by \thref{ass}. Given that $\eta$ was arbitrary, we deduce that $w \in \vec \phi_t$. Hence $\phi_t \subset \vec\phi_t$. This proves the claim.

Hence, there exists  a $\cadlag$ selection of $\cl \phi$ continuous outside $D_1$, due to Theorem \ref{lab:maintheorem2}. From this one derives the existence of a selection $\pathy_{n,\epsilon,y} \in \pathspace(\Gamma)$ with $\pathy_{n,\epsilon,y}(t) \in y + 2 \epsilon \uball$ for $t \in [a_n,b_n]$, and the continuity in $[0,T]\setminus D_1$. 

If $T \notin D_1$ it is easy to verify that the required representation holds with the countable family $\{\pathy_{n,2^{-k},y_m}\}_{n,k,m} \subset \pathspace(\Gamma)$.  Analogously, if $T \in D_1$, consider the family $\{\pathy_{j}\}_{j \in \naturals} \cup \{\pathy_{n,2^{-k},y_m}\}_{n,k,m} \subset \pathspace(\Gamma)$.
\end{proof}

%---------------------------------------------------------
\section{Examples} \label{sec:examples}
%---------------------------------------------------------
In this section we give examples of mappings having the representation \eqref{eq:MichaelRepresentation}. For the first two examples we prove directly the representation and then conclude that the mapping in the examples satisfy \thref{ass}. For the last two examples, we verify \thref{ass} and conclude the representation \eqref{eq:MichaelRepresentation}.

The mapping $\Gamma$ is said to be {\em right-outer semicontinuous} (right-osc) if  its graph is closed in the product topology of $\tau_r$ and the  usual topology on $\reals^d$. The mapping $\Gamma$ is {\em right-continuous} ($\cad$) if it is both right-isc and right-osc. Left-outer semicontinuous (left-osc) and left-continuous ($\cag$) mappings are defined analogously. We say that $\Gamma$ {\em has limits from the left} ($\lag$) if, for all $t$,
\[
\liminf_{s\upuparrows t} \Gamma_s=\limsup_{s\upuparrows t} \Gamma_s,
\]
where the limits are in the sense of \cite[Section 5.B]{rw98} and are taken along strictly increasing sequences. {\em Having limits from the right} ($\lad$) is defined analogously. A mapping $\Gamma$ is {\em \cadlag\ } (resp. {\em \caglad\ }) if it is both $\cad$ and $\lag$ (both $\cag$ and $\lad$).

In the following theorem, the distance of $x$ to $\Gamma_t$ is defined, as usual, by  
\[
d(x, \Gamma_t):=\inf_{ x' \in \Gamma_t} d(x,x'),
\]
where the distance of two points is given by the  euclidean metric.

\begin{proposition}\label{thm:cadlagGamma}
Let $\Gamma:[0,T]  \rightrightarrows \reals^d$ be  a \cadlag\ nonempty closed-convex valued mapping. For every $x\in\reals^d$, the function $\pathy$ defined by
\[
\pathy(t)=\argmin_{x'\in\Gamma_t} d(x,x')
\]
satisfies $\pathy\in \pathspace(\Gamma)$ and
\begin{equation*}
\pathy(t-)=\argmin_{x'\in\vec\Gamma_t} d(x,x').
\end{equation*}
In particular,
\[
\Gamma_t=\cl \{\pathy(t) \mid \pathy \in \pathspace(\Gamma)\}
\]
and $\vec \Gamma$ is \caglad\ nonempty convex-valued with
\[
\vec \Gamma_t=\cl \{\pathy(t-)\mid \pathy\in \pathspace(\Gamma)\}.
\]
\end{proposition}
\begin{proof}
Since $\Gamma_t$ is closed-convex valued, by the strict convexity of the distance mapping, the argmin in the definition of $\pathy$ exists and is unique; see \cite[Thm. 2.6]{rw98}. By \cite[Proposition 4.9]{rw98}, $\pathy$ is $\cad$. Take $t \in (0,T]$. On the other hand, for every strictly increasing $t^\nu\nearrow t$, $\Gamma_{t^\nu}\to \vec \Gamma_t$, so $\pathy$ is $\lag$, by \cite[Proposition 4.9]{rw98} again. 
	
Next we show
\begin{equation*}
\pathy(t-)=\argmin_{x'\in\vec\Gamma_t} d(x,x').
\end{equation*}
Since $\Gamma$ is \lag, we get that the left continuous version of $\pathy$ denoted $\pathy_-$ is a selection of $\vec\Gamma$, so the inequality $d(x, \vec \Gamma_{\bar t})\le d(x, \pathy(t-)) $ is trivial. For the other direction, assume for a contradiction that  $d(x, \vec \Gamma_{\bar t})<d(x, \pathy(\bar t-)) $ for some $\bar t \in (0,T]$. There is $s<\bar t$ such that  $d(x, \vec \Gamma_{\bar t})<d(x, \pathy(s')) $ for all $s'\in(s,\bar t)$. By the definition of $\vec\Gamma$, this means that 
\[
\pathy(s')\notin\argmin_{x'\in\Gamma_{s'}} d(x,x')
\]
for some $s'\in(s,\bar t)$, which is a contradiction. The claims $\Gamma_t=\cl \{\pathy(t) \mid \pathy \in \pathspace(\Gamma)\}$ and $\vec \Gamma_t=\cl \{\pathy(t-)\mid \pathy\in \pathspace(\Gamma)\}$ are now immediate while $\vec\Gamma$ is \caglad\ due to \cite[Exercise 4.2]{rw98}.
\end{proof}

For the next example recall that $\Gamma$ is {\em solid} if for each $t \in [0,T]$ the set $\Gamma_t$ is equal to the closure of its interior. For a closed-convex valued mapping, this property is equivalent to $\interior \Gamma_t\ne\emptyset$ for all $t \in [0,T]$; see \cite[Example 14.7]{rw98}.  Recall also our convention that for a function $\pathy$ we set $\pathy(0-)=0$.

\begin{proposition}\label{thm:solidGamma}
Let $\Gamma:[0,T]  \rightrightarrows \reals^d$ be a closed convex-valued solid mapping with full domain and $\tau_r$-isc. Assume that $\vec \Gamma$ has full domain. If  $\vec \Gamma$ is also solid then $\Gamma$ has a representation \eqref{eq:MichaelRepresentation}.  In this case
\[
\vec \Gamma_t =\cl \{\pathy(t-) \mid \pathy \in \pathspace(\Gamma)\}, \hspace{.1cm} t\in [0,T].
\]
\end{proposition}
\begin{proof}
We first show that $\pathspace(\Gamma)\ne\emptyset$. For $\bar t \in [0,T)$ and $y^r  \in \interior \Gamma_{\bar t}$, there exists $\delta>0$ such that $y^r \in \Gamma_u$ for $u \in [\bar t,\bar t+\delta]$ since $\Gamma$ is $\tau_r$-isc and solid. Indeed, let $\bar y^i$ be a finite set of points in $\interior \Gamma_{\bar t}$ such that $y^r$ belongs to the interior of the convex hull $\co\{\bar y^i\}$. Let $\epsilon>0$ be small enough so that $y^r \in \co\{ v^i\}$ whenever, for every $i$, $v^i\in (\bar y^i+\epsilon \uball)$. Since $\Gamma$ is $\tau_r$-isc, there is, for every $i$, $u^i>\bar t$ such that $\Gamma_u\cap(\bar y^i+\epsilon \uball)\ne\emptyset$ for every $u\in[\bar t,u^i)$. Denoting $\bar u=\min u^i$, we have, by convexity of $\Gamma$, that 
\begin{equation}\label{eq:kp}
y^r\in \Gamma_u \mbox{ for every } u\in[\bar t,\bar u).
\end{equation}
Now assume $\bar t>0$ and  take $y^l  \in \interior \vec \Gamma_{\bar t}$.  We now show the existence of $s<\bar t$ such that
\begin{equation}\label{eq:kpleft}
y^l \in \Gamma_u \mbox{ for every } u\in(s,\bar t].
\end{equation}
Assume for a contradiction the existence of $t^\nu \nearrow \bar t$ such that $y^l \notin \Gamma_{t^\nu}$.  Let $\bar y^i \in \interior \vec \Gamma_{\bar t}$ be $d+1$ points and $\epsilon>0$  such that for any points $\tilde y^i \in \bar y^i  + \epsilon \uball$, $\tilde y^i \in  \vec \Gamma_{\bar t}$ and $y^l \in \co\{\tilde y^i \}$.  By the definition of $\vec \Gamma$ as a left-limit, there exists $\nu_0 \in \naturals$ such that for all $\nu> \nu_0$ there exists $y^i_{\nu} \in \Gamma_{t^\nu}$ with $y^i_{\nu} \in \bar y^i  + \epsilon \uball$ for $i=1,\ldots,d+1$. Then,  $y^{l} \in \co\{y^i_{\nu}\}$ and this last set is included in $\Gamma_{t^\nu}$ by convexity. Then, $y^{l} \in \Gamma_{t^\nu}$, a contradiction.
 
After the preliminary preparations showing the existence of \eqref{eq:kp} and \eqref{eq:kpleft}, we construct local selections of $\Gamma$ that can be pasted together through a partition of unity as we have done in the proof of Lemma \ref{labLemma1} and of Theorem \ref{lab:maintheorem2}.
	
To prove $\vec \Gamma_t =\cl \{\pathy(t-) \mid \pathy \in \pathspace(\Gamma)\}$, the inclusion $\supseteq$ is clear. Now take $y^l \in \interior \vec \Gamma_{\bar t}$ and $s<\bar t$ as in  \eqref{eq:kpleft} and $\pathy \in \pathspace(\Gamma)$. Defining
\[
\begin{cases}
y^l \quad & t \in [s, \bar t)\\
\pathy(t) \quad & otherwise,
\end{cases}
\]
we get the remaining inclusion.
\end{proof}

In the next examples we verify that \thref{ass} is satisfied.  Then the representation \eqref{eq:MichaelRepresentation} holds by Theorem \ref{lab:maintheorem1}.
\begin{proposition}
Let $f:[0,T] \to \reals$ be continuous with respect to $\tau_r$. Assume $f > 0$. Let $\Gamma$ be the mapping defined by $\Gamma_t:= [0,f(t)]$. Then $\Gamma$ satisfies \thref{ass}.
\end{proposition}
\begin{proof}
Let $O$ be an open set with $\Gamma_s \cap O \neq \emptyset$ for $s \in (t-\delta,t)$. Let $a^*:=\inf(O \cap \reals_+)$. It is clear that  for $s  \in (t-\delta,t)$ we have $a^* < f(s)$, otherwise $\Gamma_s\cap O =\emptyset$. From the inequality is easy to verify that $a \in \overrightarrow{\Gamma \cap O}$ at $t$. Indeed, take an increasing sequence $z_n \nearrow t$. For each $n$, there exists $y_{z_n} \in (a^*,f(z_n)) \cap (a^*,a^*+(t-z_n))\cap O$. Hence $y_{z_n} \to a^*$ showing that $a^* \in \liminf_{n}(\Gamma_{z_n} \cap O)$. 
\end{proof}

\begin{proposition}
Let $f,g:[0,T] \to \reals$ be two $\cadlag$ functions  with $f \geq g$. Let $\Gamma$ be the mapping defined by $\Gamma_t:= [g(t),f(t)]$. Then $\Gamma$ satisfies \thref{ass}.
\end{proposition}
\begin{proof}
We only show that $\Gamma$ satisfies \thref{ass}, since  the other assumptions are easier to verify. Let $O$ be an open set with $\Gamma_s \cap O \neq \emptyset$ for $s \in (t-\delta,t)$. Let $a^*:=\inf (\cl O \cap [g(t-),f(t-)])$ and $a_s:=\inf(O \cap \Gamma_s)$.  We claim that $\lim_{s\nearrow t} a_s=a^*$ from which the proposition follows.

The set $\cl O \cap [g(t-),f(t-)]$ is easily seen to be non empty, so $a^*$ is well defined. Now we verify that $\liminf a_s \geq a^*$. Take $y_s \in O\cap \Gamma_s \cap [a_s,a_s + t-s]$. Then, $g(t-) \leq\liminf y_s \leq \liminf a_s$. Moreover, $\limsup a_s = \limsup y_s  \leq \limsup f(s)$.  It is clear that $a_s \in \cl O$ so $\liminf a_s \in \cl O$. Hence, $\liminf a_s \in \cl O \cap [g(t-),f(t-)]$ and then $a^* \leq \liminf a_s$.  Now we show that $a^* \geq \limsup a_s$.  There are two cases, in the first $a^*=f(t-)$.  Then, $a_s \leq f(s)$ and $\limsup a_s \leq f(t-)=a^*$.  In the second, $a^* < f(t-)$. Let $\epsilon_0>0$ be such that for $\epsilon \in (0,\epsilon_0)$   we have $a^* < f(t-)-\epsilon$ and $g(t-) + \epsilon < f(t-)  -\epsilon$ and
\[
\exists y^{\epsilon} \in O \cap [g(t-) + \frac{1}{2}\epsilon, f(t-)-\frac{1}{2}\epsilon]\cap [a^*,a^*+\epsilon].
\]
Now fix $\epsilon \in (0,\epsilon_0)$. We easily verify that $y^{\epsilon} \in \Gamma_s$ for $s \in (t-\eta,t)$ where $\eta$ is such that $|f(s)-f(t-)|<\epsilon/3$ and $|g(s)-g(t-)|<\epsilon/3$. Hence $y^{\epsilon} \in \Gamma_s \cap O$ implying that $a_s \leq y^{\epsilon} \leq a^*+\epsilon$.
\end{proof}

\section{An application to integral functionals}\label{lab:sectionApp}
In this section we develop an interchange rule for integral functionals of cadlag functions whics builds on the  the representation \eqref{eq:MichaelRepresentation}. Interchange rules go back to the seminal paper of \cite{roc68} in decomposable spaces of $\reals^d$-valued measurable functions and are fundamental in obtaining convex duality in calculus of variations and optimal control; see, e.g., \cite{roc78} or \cite{pp14} for a more recent application . The interchange rule proved here is a starting point for the companion paper \cite{pt21a}  where integral functionals of $\cadlag$ stochastic processes are analyzed in detail. Further applications are given in the follow-up papers \cite{pt22a,pt22b}.

A function $h : [0,T]  \times \reals^d \to \ereals$ is a \textit{normal integrand} on $[0,T]$ if its {\em epigraphical mapping} $\epi h: [0,T] \tos  \reals^d \times \reals$ defined by
\[
\epi h (t) := \{(x,\alpha) \in   \reals^d \times \reals \mid  h (t,x) \leq \alpha\}
\]
is closed-valued and measurable. When this mapping is also convex-valued, $ h $ is  a \textit{convex normal integrand}. A general treatment of normal integrands on $\reals^d$ can be found from \cite[Chapter 14]{rw98} while integrands on a Suslin space are systematically presented in \cite{cv77}. In particular, a normal integrand $h$ is jointly measurable so that the {\em integral functional} with respect to a nonnegative Radon measure $\mu$ on $[0,T]$ given by 
\[
K(\pathy):=\int_{[0,T]} h(t, \pathy(t))\mu(dt)
\]
is well-defined for any measurable $\pathy:[0,T]\to \reals^d$. As usual in convex analysis,  an integral is defined as $+\infty$ unless the positive part is integrable. For a normal integrand $ h : [0,T] \times \reals^d \to \ereals$ the {\em domain mapping } is defined by  $\dom h_t=\{y \in \reals^d \mid h(t,y)<\infty\}$, its image closure is 
\begin{equation*}
S_{t} :=\{y \in \reals^d  \mid  y \in\cl \dom h(t,\cdot)\}
\end{equation*}
and
\begin{equation*}
\pathspace(S):=\{\pathy \in \pathspace \mid  \pathy(t)\in S_t\ \forall t \in [0,T] \}
\end{equation*}
is the set of {\em \cadlag\ selections} of $S$.  

In the next assumption we collect the necessary conditions  in order to obtain the interchange rule of Theorem \ref{thm:ifcadlag}. In particular, we require a representation of the mapping $S$ in terms of its $\cadlag$ selections. As we have know, Theorem \ref{lab:maintheorem1} gives necessary and sufficient conditions for the mapping $S$ to have such a representation.
\begin{assumption}\thlabel{ass:clS}
Assume $\mu$ is a  nonnegative Radon measure and $h$ is a convex normal integrand $h:[0,T] \times \reals^d \to \ereals$ such that 
\begin{align}\thlabel{ass:clS1}
S_t&= \cl \{\pathy(t) \mid \{\pathy \in \pathspace(S)\}\quad\forall\ t,\\
\pathspace(S) &=\cl(\dom K\cap\pathspace(S)),\nonumber
\end{align}
where the latter closure is with respect to pointwise convergence. 
This means that for each $\pathy \in \pathspace(S)$ there exists a sequence $\{\pathy_{\nu} \}_{\nu=1}^{\infty} \subset \dom K\cap \pathspace(S)$ converging pointwise to $\pathy$. 
\end{assumption}

The following theorem is variant of the main theorem in \cite{per18a} that established a similar interchange rule for integral functionals of continuous functions. In that context, the first condition in \thref{ass:clS} is simply the original Michaels representation while the second condition is analyzed in detail in \cite{per14}.
\begin{theorem}\label{thm:ifcadlag}
Under Assumption \ref{ass:clS},
\[
\inf_{\pathy\in \pathspace(S)} \int_{[0,T]}h(t,\pathy(t))\mu(dt)  = \int_{[0,T]} \inf_{y\in\reals^d} h(t,y)\mu(dt)
\]
as soon as the left side is less than $+\infty$.
\end{theorem}
\begin{proof}
We have $\pathspace(S)=\cl(\dom K^{\mu}_h\cap \pathspace(S))$ and $\pathspace(S)$ is PCU-stable in the sense of \cite{bv88}, so, by \cite[Theorem~1]{bv88},
\[
\inf_{\pathy\in \pathspace(S)} \int_{[0,T]}h(t,\pathy(t))\mu(dt)  = \int_{[0,T]} \inf_{y\in\Gamma_t} h(t,y)\mu(dt), 
\]
where $\Gamma$ is the essential supremum of $\pathspace(S)$, i.e., the smallest (up to a $\mu$-null set) closed-valued mapping for which every $\pathy\in \pathspace(S)$ is a selection of $\Gamma$ $\mu$-almost everywhere. It remains to show that $S_t\subseteq\Gamma_t$ $\mu$-almost everywhere, since then  the infimum over $\Gamma_t$ can be taken instead over all of $\reals^d$. 
	
It suffices to show that, for every closed ball $\cl \uball_\nu$ with radius $\nu=1,2,\dots$, we have $S_t\cap \cl \uball_\nu\subseteq\Gamma_t\cap \cl \uball_\nu$ $\mu$-almost everywhere. Thus we may assume that $\Gamma$ and $S$ are compact-valued. Assume for a contradiction that $B=\{t\mid S_t\not \subset \Gamma_t\}$ satisfies $\mu(B)>0$. Let $H_{v,b} :=\{x\in\reals^d\mid v\cdot x < b\}$ and $B_{v,b}:=\{t\mid S_t\cap H_{v,b}\ne\emptyset, H_{y,b}\cap \Gamma_t =\emptyset\}$. Since $S$ and $\Gamma$ are compact convex-valued, 
\[
B=\bigcup_{(v,b)\in\mathbb Q^d\times \mathbb Q} B_{v,b}.
\]
Since $\mu(B)>0$, and rationals vectors are countable, there exists $(v,b)\in\mathbb Q^d\times \mathbb Q$ such that $\mu (B_{v,b})>0$.  Since $\mu$ is Radon, passing to a subset if necessary, we may assume that $B_{v,b}$ is closed and still of positive measure. For every $t\in B_{v,b}$ there exists $\pathy^t\in \pathspace(S)$ such that $\pathy^t(t)\in S_t\cap H_{v,b}$. By right-continuity, there exists $\delta(t)$ such that 
\[
\pathy^t(t')\in H_{v,b}\quad\forall t'\in[t,t+\delta(t)).
\]
Now $\{[t,t+\delta(t))\mid t\in B_{v,b}\}$ is an open cover of $B_{v,b}$ in the $B_{v,b}$-relative right half-open topology of $[0,T]$. The topology $\tau_r$  clearly is separable and it is shown in \cite{Sorgenfrey1947} to be paracompact. Hence it is Lindel\"of by \cite[Theorem VIII.7.4]{Dugundji1966}. Since $B_{\nu,q}$ is closed and the topology $\tau_r$ is Lindel\"of, there exists a countable subcover $\{[t^\nu,t^\nu+\delta(t^\nu))\mid t^\nu\in B_{v,b}\}$. Since $\mu(B_{q,v})>0$, $\sigma$-additivity of $\mu$ implies that  $\mu([t^\nu,t^\nu+\delta(t^\nu))>0$ for some $\nu$. But, for every $t'\in [t^\nu,t^\nu+\delta(t^\nu))$,  $\pathy^{t^\nu}(t')\notin \Gamma_{t'}$, which contradicts that every $\pathy\in \pathspace(S)$ is a selection of $\Gamma$ $\mu$-almost everywhere.
\end{proof}

%\appendix
%---------------------------------------------------------
\section{Appendix} \label{sec:others}
%---------------------------------------------------------
The following result is a special case of \cite[Proposition 2.5]{mic56}.
\begin{lemma}\label{lemma:iscStab}
Let  $\Gamma^1, \Gamma^2: [0,T] \rightrightarrows  \reals^d$ be $\tau_r$-isc mappings. Let $\epsilon>0$. Then, the mapping $\Gamma = \Gamma^1 \cap (\Gamma^2 + \epsilon \uball)$ is also $\tau_r$-isc.
\end{lemma}
\begin{proof}
Let $O$ be an open set of $\reals^d$, then
\[
\begin{split}
\Gamma^{-1}(O)
&= \{ t \in [0,T] \mid O \bigcap [ \Gamma^1(t) \cap (\Gamma^2(t) + \epsilon \uball)] \neq \emptyset \}\\
&=  \{ t \in [0,T] \mid \Gamma^2 \times\Gamma^1   \cap  [ W \cap (\reals^d \times O)] \neq \emptyset \},
\end{split}
\]
where $W=\{(x,y) \in \reals^d \times \reals^d \mid x-y \in \epsilon \uball \}$.
Hence $\Gamma^{-1}(O)$ is $\tau_r$-open.
\end{proof}

\bibliographystyle{plain}
%\bibliography{/home/er/Dropbox/Investigacion/commonLatex/bibliography14}
\bibliography{sp2}

\def\cprime{$'$}
\begin{thebibliography}{10}

\bibitem{bv88}
G.~Bouchitt{\'e} and M.~Valadier.
\newblock Integral representation of convex functionals on a space of measures.
\newblock {\em J. Funct. Anal.}, 80(2):398--420, 1988.

\bibitem{cv77}
C.~Castaing and M.~Valadier.
\newblock {\em Convex analysis and measurable multifunctions}.
\newblock Springer-Verlag, Berlin, 1977.
\newblock Lecture Notes in Mathematics, Vol. 580.

\bibitem{Dugundji1966}
James Dugundji.
\newblock {\em Topology}.
\newblock Allyn and Bacon series in advanced mathematics. Allyn and Bacon,
  1966.

\bibitem{mic56}
E.~Michael.
\newblock Continuous selections. {I}.
\newblock {\em Ann. of Math. (2)}, 63:361--382, 1956.

\bibitem{pp14}
T.~Pennanen and A.-P. Perkki\"o.
\newblock Duality in convex problems of {B}olza over functions of bounded
  variation.
\newblock {\em SIAM Journal on Control and Optimization}, 52:1481--1498, 2014.

\bibitem{pp18d}
T.~Pennanen and A.-P. Perkki\"{o}.
\newblock Convex integral functionals of regular processes.
\newblock {\em Stochastic Process. Appl.}, 128(5):1652--1677, 2018.

\bibitem{per14}
A.-P. Perkki{\"o}.
\newblock Continuous {E}ssential {S}elections and {I}ntegral {F}unctionals.
\newblock {\em Set-Valued Var. Anal.}, 22(1):45--58, 2014.

\bibitem{per18a}
A.-P. Perkki{\"o}.
\newblock Conjugates of integral functionals on continuous functions.
\newblock {\em Journal of Mathematical Analysis and Applications}, 459(1):124
  -- 134, 2018.

\bibitem{pt21a}
A.-P. Perkki\"o and E.~Trevino~Aguilar.
\newblock {Convex integral functionals of cadlag processes}.
\newblock {\em Submitted}, 2021.

\bibitem{pt22a}
A.-P. Perkki\"o and E.~Trevino~Aguilar.
\newblock {Convex duality for partial hedging of American options: Continuous
  price processes}.
\newblock {\em Submitted}, 2022.

\bibitem{pt22b}
A.-P. Perkki\"o and E.~Trevino~Aguilar.
\newblock {Duality in irreversible cumulative investment}.
\newblock {\em Submitted}, 2022.

\bibitem{roc68}
R.~T. Rockafellar.
\newblock Integrals which are convex functionals.
\newblock {\em Pacific J. Math.}, 24:525--539, 1968.

\bibitem{roc78}
R.~T. Rockafellar.
\newblock Duality in optimal control.
\newblock In {\em Mathematical control theory ({P}roc. {C}onf., {A}ustralian
  {N}at. {U}niv., {C}anberra, 1977)}, volume 680 of {\em Lecture Notes in
  Math.}, pages 219--257. Springer, Berlin, 1978.

\bibitem{rw98}
R.~T. Rockafellar and R.~J.-B. Wets.
\newblock {\em Variational analysis}, volume 317 of {\em Grundlehren der
  Mathematischen Wissenschaften [Fundamental Principles of Mathematical
  Sciences]}.
\newblock Springer-Verlag, Berlin, 1998.

\bibitem{Sorgenfrey1947}
R.~H. Sorgenfrey.
\newblock {On the topological product of paracompact spaces}.
\newblock {\em Bulletin of the American Mathematical Society}, 53(6):631 --
  632, 1947.

\end{thebibliography}

\end{document}